\documentclass[reqno]{amsart}

\usepackage{amssymb,latexsym,mathrsfs,
	amsxtra,amscd,amsfonts,amsthm,amsmath,verbatim,epsfig}

\usepackage{color}

\usepackage[colorlinks,pagebackref=true]{hyperref}

\usepackage[english]{babel}

\newcommand{\arxiv}[1]{\href{http://arxiv.org/abs/#1}{{\tt arXiv:#1}}}

\makeatletter

\makeatletter
\def\widebreve#1{\mathop{\vbox{\m@th\ialign{##\crcr\noalign{\kern3\p@}%
				\brevefill\crcr\noalign{\kern3\p@\nointerlineskip}%
				$\hfil\displaystyle{#1}\hfil$\crcr}}}\limits}
\def\brevefill{$\m@th \setbox\z@\hbox{$\braceld$}%
	\bracelu\leaders\vrule \@height\ht\z@ \@depth\z@\hfill\braceru$}

\makeatletter

\def\@citecolor{blue}
\def\@linkcolor{red}
\def\@urlcolor{blue}
\def\@urlcolor{blue}

\numberwithin{equation}{section}
\def\im{\operatorname{im}}

\def\ker{\operatorname{ker}}
\def\rk{\operatorname{rank}}
\def\syz{\operatorname{Syz}}
\def\tor{\operatorname{Tor}}
\def\soc{\operatorname{Soc}}
\def\dim{\operatorname{dim}}
\def\depth{\operatorname{depth}}

\def\pd{\operatorname{pd}}
\def\grade{\operatorname{grade}}
\def\ass{\operatorname{Ass}}

\def\ann{\operatorname{Ann}}
\def\spec{\operatorname{Spec}}

\newcommand{\mm}{\mathfrak m}

\newcommand{\lm}{{\lambda}}

\newcommand{\bl}{\begin{lemma}}
	\newcommand{\el}{\end{lemma}}
\newcommand{\bt}{\begin{theorem}}
	\newcommand{\et}{\end{theorem}}
\newcommand{\ben}{\begin{enumerate}}
	\newcommand{\een}{\end{enumerate}}
\newcommand{\bpf}{\begin{proof}}
	\newcommand{\eepf}{\end{proof}}
\newcommand{\beqn}{\begin{eqnarray*}}
	\newcommand{\eeqn}{\end{eqnarray*}}
\newcommand{\beqnn}{\begin{eqnarray}}
\newcommand{\eeqnn}{\end{eqnarray}}
\newcommand{\bd}{\begin{definition}}
	\newcommand{\ed}{\end{definition}}
\newcommand{\bp}{\begin{proposition}}
	\newcommand{\ep}{\end{proposition}}
\newcommand{\bc}{\begin{corollary}}
	\newcommand{\ec}{\end{corollary}}
\newcommand{\bex}{\begin{example}}
	\newcommand{\eex}{\end{example}}

\newcommand{\wlg}{ Without loss of generality }

\theoremstyle{plain}
\newtheorem{theorem}{Theorem}[section]
\newtheorem{corollary}[theorem]{Corollary}
\newtheorem{proposition}[theorem]{Proposition}
\newtheorem{lemma}[theorem]{Lemma}

\newtheorem{example}[theorem]{Example}
\newtheorem{definition}[theorem]{Definition}

\newtheorem{question}[theorem]{Question}

\theoremstyle{remark}
\newtheorem{remark}[theorem]{Remark}
\numberwithin{equation}{theorem}

\def\im{\operatorname{im}}
\def\ker{\operatorname{ker}}

\begin{document}
\title{Frobenius Betti numbers and syzygies of finite length modules}
\author{Ian M. Aberbach \and Parangama Sarkar}
\address{Ian M. Aberbach, Department of Mathematics,
	University of Missouri, Columbia, MO 65211, USA}
\email{aberbachi@missouri.edu}

\address{Parangama Sarkar, Department of Mathematics,
	University of Missouri, Columbia, MO 65211, USA}
\email{parangamasarkar@gmail.com}
\subjclass{Primary: 13D02; Secondary: 13A35, 13H99}
\keywords{Finite length syzygies, Frobenius Betti numbers, projective dimension, phantom homology, local cohomology}


\date{\today}

\begin{abstract}
Let $(R,\mm)$ be a local (Noetherian) ring of dimension $d$ and $M$ a finite length $R$-module with free resolution $G_\bullet$.  De Stefani, Huneke, and N\'{u}\~{n}ez-Betancourt explored two questions about the properties of resolutions of $M$.    First, in  characteristic $p>0$, what vanishing conditions on the Frobenius Betti numbers, $\beta_i^F(M, R) : = \lim_{e \to \infty} \lambda(H_i(F^e(G_\bullet)))/p^{ed}$,  force $\pd_R M < \infty$.  Second, if $\pd_R M = \infty $, does  this force $d+2$nd or higher syzygies of $M$ to have infinite length.

For the first question, they showed, under rather restrictive hypotheses, that $d+1$ consecutive vanishing Frobenius Betti numbers forces $\pd_R M < \infty$.  And when $d=1$ and $R$ is CM then one vanishing Frobenius Betti number suffices.  Using properties of stably phantom homology, we show that these results hold in general, i.e., $d+1$ consecutive vanishing Frobenius Betti numbers force $\pd_R M < \infty$, and, under the hypothesis that $R$ is CM, $d$ consecutive vanishing Frobenius Betti numbers suffice.

For the second question, they obtain very interesting results when $d=1$.  In particular, no third syzygy of $M$ can have finite length.  Their main tool  is, if $d=1$, to show, if the syzygy has a finite length, then it is an alternating sum of lengths of Tors.   We are able to prove this fact for rings of arbitrary dimension, which allows us to show that if $d=2$, no third syzygy of $M$ can be finite length!  We also are able to show that the question has a positive answer if the dimension of the socle of $H^0_{\mm}(R)$ is large relative to the rest of the module, generalizing the case of Buchsbaum rings.
\end{abstract}

\maketitle

\section{Introduction}
Let $(R, \mm, k)$ be a commutative local (so Noetherian) ring.  When $M$ is an $R$-module of finite length, we may consider various hypotheses about the syzygies of $M$, or of related modules, and ask if these hypotheses suffice to show that the projective dimension of $M$ is finite.  In particular, we are motivated by two questions asked by De Stefani, Huneke and N\'{u}\~{n}ez-Betancourt in \cite{SHB}.

The first question is in the context of local rings of positive prime characteristic $p$, where we have the Frobenius endomorphism $f: R \to R$ sending $r \mapsto r^p$, and its iterates $f^e:R \to R$ for $e \ge 0$.  Let $\lambda(M)$ denote the length of the module $M$.  If $\lambda(M) < \infty$, and $G_\bullet$ is a resolution of $M$ by finitely generated free modules, then $F^e(G_\bullet)$ (see section 2 for background on the Frobenius functor) has finite length homology for all all $e\ge 0$.  If $\dim(R) =d$, we may consider the Frobenius Betti numbers 
$$\beta_i^F(M,R) = \lim\limits_{e \to \infty} \dfrac{\lambda(H_{i}(F^e(G_\bullet)))} {p^{ed}},
$$
 which are non-negative real numbers.  Having the extremal value of $0$ for some positive $i$ should suggest that $M$ is particularly well-behaved.  Indeed, Miller showed in \cite[Corollary 2.5]{Mi} that if $R$ is a complete intersection then the existence of even one $\beta^F_i(M,R) = 0$ implies that $\pd_R M < \infty$.  

De Stefani, Huneke, and N\'{u}\~{n}ez-Betancourt posed the question \cite{SHB}:
\begin{question}\label{1st question}
Let $M$ be an $R$-module of finite length.  What vanishing conditions on $\beta_i^F(M,R)$ imply that $M$ has finite projective dimension?
\end{question}
Technically, they asked the question only in the case that $R$ is F-finite (see section 2), but, in the manner $\beta_i^F$ is defined here, that hypothesis is not necessary.

They then showed that if $R$ has a finitely generated regular algebra of the same dimension and if $d+1$ consecutive $\beta_i^F(M,R)$ (with $i$ positive) vanish, then $\pd_R M < \infty$ (see \cite[Proposition 4.1]{SHB}).  Moreover, if $\dim R =1$ and $R$ is Cohen-Macaulay, then a single vanishing $\beta_i^F(M,R)$ (for $i >0$) suffices (\cite[Corollary 4.8]{SHB}).

We are able to generalize both of these results.  The former result is true without any hypotheses on the ring $R$, while the latter result is true in any arbitrary positive dimension provided only that $R$ have positive depth and is formally equidimensional (a substantially weaker hypothesis than being CM).  See Corollary~\ref{general} and Corollary~\ref{result}.

The second way in which a module $M$ of finite length may be ``close to'' being finite projective dimension is if $M$ has a syzygy of finite length.  For any local ring $(R, \mm)$ of dimension one and depth zero, if we take a parameter $x \in \mm$ then $M= R/xR$ is a module which is not of finite projective dimension, but which has a second syzygy of finite length.  Thus, one should ask about the possibility of syzygies of finite length at the $\dim(R)+2$ spot or higher (\cite{A} shows that the $i$th syzygies for $0 < i \leq \dim(R)$ syzygies are not finite length).  In light of these facts, De Stefani, Huneke, and N\'{u}\~{n}ez-Betancourt posed the question \cite{SHB}:

\begin{question}\label{2nd question}
Let $R$ be a $d$-dimensional local ring, and let $M$ be a finitely generated $R$-module such that $\pd_R(M) = \infty$ and $\lambda(M) < \infty$.  If $i > d+1$, then must the length of the $i$th syzygy be infinite?
\end{question}

There are very few satisfying results in this direction, although in dimension one it is shown in \cite{SHB} that if a module of finite length has an $i$th syzygy, $\Omega_i$, of finite length, then $\lambda(\Omega_i)$ may be computed as an alternating sum of lengths of certain Tor modules.  In this way, they are able to show that no such $\Omega_3$ may be finite length.  We give a simpler proof of their result on the length of such an $\Omega_i$, which holds in all dimensions (see Proposition~\ref{extended}).  

There are several interesting results that follow from Proposition~\ref{extended}.  We show, in Theorem~\ref{dimension two}, that in a ring of dimension two, no {\it third} syzygy of a finite length module can have finite length, suggesting that for $d \ge 2$, Question~\ref{2nd question} may have a positive answer for $i \ge d+1$ (not just for $i > d+1$).  We also show that one form of bad behavior with respect to Question~\ref{2nd question} forces good behavior in other cases with respect to Question~\ref{2nd question}.  Specifically, in Theorem~\ref{bad to good}, we show that if $H^0_\mm(R)$ has an $i-2$nd syzygy of finite length for some $i \ge 4$, then $\syz_{i+1}(M)$, where $M$ has finite length, cannot have finite length.

We also show that Question~\ref{2nd question} has a positive answer for rings in which the socle dimension of $H^0_\mm(R)$ is large relative to the total length of $H^0_\mm(R)$, generalizing the case of Buchsbaum rings done in \cite{SHB}.  See Theorem~\ref{big socle} for the precise statement.
\subsection*{Acknowledgements.}
The second author would like to express her sincere gratitude to Olgur Celikbas for suggesting the paper \cite{SHB}. The authors would like to thank Alessandro De Stefani for interesting and useful comments. The second author was supported by IUSSTF, SERB Indo-U.S. Postdoctoral Fellowship 2017/145 and DST-INSPIRE India.
\section{Brief tight closure background}

We will briefly outline what we need with regard to tight closure below and  refer the interested reader to the sources \cite{HH}, and \cite{HH2} for more information.

Let $R$ be a Noetherian ring.  In section~\ref{FBnumber} we will be dealing exclusively with rings of positive prime characteristic $p$, in which case we have the Frobenius endomorphism $f:R \to R$ sending $r$ to $r^p$ and its iterates.  It is often helpful to denote the target by ${}^1\! R$.  Similarly, we have iterates
$f^e: R \to {}^e\! R = R$.  For each $e$ there is a functor from the category of $R$-modules to itself obtained by tensoring with ${}^e\!R$ and then identifying ${}^e\!R$ with $R$.  We denote this functor by $F^e$.  In particular $F^e(M) := {}^e \!R \otimes_R M$, where we have  $r \otimes am = a^qr \otimes m$ and $b(r \otimes m) = (br)\otimes M$.  

Using $q = p^e$, for an ideal $I \subseteq R$, we set $I^{[q]} = (a^q : a \in I)$, and we observe that $I^{[q]}$ is $I \,{}^e\! R$ so that $F^e(R/I) = R/I^{[q]}$.  More generally, for $N \subseteq M$ we set $N^{[q]}_M = \im(F^e(N) \to F^e(M))$.   We also use the notation that for $m \in M$, $m^q = 1 \otimes m \in F^e(M)$.   By the right exactness of tensor, if a module has presentation $R^{b_1} \overset {\phi} \to R^{b_0}$ and $\phi$ can be represented by the $b_0 \times b_1$ matrix $[a_{ij}]$, then $F^e(M)$ is presented by the matrix $[a^q_{ij}]$.  We will be particularly concerned with left complexes of finitely generated free modules, $(G_\bullet, \phi_\bullet)$, in which case, applying the Frobenius functor gives $(F^e(G_\bullet), F^e(\phi_\bullet))$, which is the left complex with free modules of the same rank, and matrices with entries in each position raised to the $q$th power.

 We let $R^o$ be the complement of the minimal primes of $R$ (e.g., $R -\{0\}$ in the case that $R$ is a domain).  Given modules $N \subseteq M$ we say that the element $m \in M$ is in the tight closure of $N$ in $M$, denoted, $N^*_M$, if there exists a $c \in R^o$ such that for all $q \gg 0$, $cm^q \in N^{[q]}_M$.  This definition is particularly easy to understand in the case that $M = R^t$ is a free module.  Then $N$ is generated by a set of $t \times 1$ vectors, say $\mathbf{n}_1, \ldots \mathbf{n}_h$ and each $\mathbf{n}_j^q$ is the vector with components raised to the $q$th power.  
 
 Let $(G_\bullet, \phi_\bullet)$ be a complex of $R$-modules (in our case $G_\bullet$ will consist of finitely generated free modules, but this is not needed for the definition). We say that the complex has {\it phantom homology at the $i$th spot} if $\ker(\phi_i) \subseteq (\im(\phi_{i+1}))^*_{G_i}$, and the complex has {\it stably phantom homology at the $i$th spot} if for all $e \ge 0$, $F^e(G_\bullet)$ has phantom homology at the $i$th spot.  If $G_\bullet$ is a left complex (i.e., $G_i = 0$ for $i < 0$), then we say that $G_\bullet$ is stably phantom acyclic if $G_\bullet$ has stably phantom homology at the $i$th spot for all $i >0$.  (Note that a free resolution of a module has phantom homology at the $i$th spot for all $i>0$, but, since the Frobenius endomorphism is not usually exact, it is very rare for such a complex to be stably phantom acyclic.  Much of the work in section \ref{FBnumber} depends on proving that enough consecutive stably phantom homologies implies that a minimal free resolution  is a finite resolution).  
 
 When testing tight closure, in principle the element $c \in R^o$ can change, but it is often the case that all tight closure tests can be done with one element.  If we know that $c \in R^o$ works for all tight closure tests then we call $c$ a {\it test element}.  The theory of test elements is extremely interesting, but is not addressed in this paper.  For our purposes it suffices to know the main theorem of \cite{HH2} on test elements: Let $R$ be essentially of finite type over an excellent local ring and reduced (e.g., complete and reduced).  If $d \in R^o$ is such that $R_d$ is regular (such elements always exist in this case), then $d$ has a power which is a test element.

We say that $R$ is   {\it F-finite} if ${}^1\!R$ is finitely generated as an $R$-module, in which case all ${}^e\! R$ are finitely generated.  More generally, if $M$ is an $R$-module then we can consider the module ${}^e\!M$ via restriction of scalars with respect to $f^e$ (this is true whether or not $R$ is F-finite).  When $R$ is F-finite and $M$ is finitely generated, then all ${}^e\! M$ are also finitely generated.  

\begin{definition}
(see section 3 of \cite{SHB}) Let $(R, \mm, k)$ be an $F$-finite local ring of characteristic $p$ and dimension $d$.  Then $\alpha := \log_p [k^{1/p}:k]$ is finite.  Let $M$ be a module of finite length and let $N$ be a finitely generated module.  For $i \ge 0$, define the
$i$th Frobenius Betti number of $N$ with respect to $M$ by
$$
\beta_i^F(M,N) = \lim_{e \to \infty} \dfrac{ \lambda(\tor_i^R(M, {}^e\!N))} {q^{d+\alpha}}.
$$
\end{definition}
We will be most interested in the case that $N = R$.  In this case, we observe that if $G_\bullet$ is a free resolution of $M$ then $\tor_i^R(M, {}^e\!R)$ can be naturally identified with $H_i(F^e(G_\bullet))$.  

We can, in fact, define Frobenius Betti numbers, for any local ring of characteristic $p$.  Let $G_\bullet$ be a resolution of the finite length module $M$ by finitely generated free modules and set
$$
\beta_i^F(M,N) = \lim_{e \to \infty} \dfrac{ \lambda(H_i(F^e(G_\bullet) \otimes_R N))} {q^{d}}.
$$
The limit exists by the results of Seibert (\cite{S}).   The complex $G_\bullet \otimes_R {}^e\! N$ can be identified with the complex $F^e_R(G_\bullet) \otimes_R N$ (since $R$ and ${}^e\!R$ are isomorphic as rings) with the length calculation adjusted by a factor of $q^{\alpha}$, i.e., $\lambda_R(H_i(G_\bullet \otimes_R {}^e\! N)) = q^{\alpha} \lambda_R(F^e(G_\bullet) \otimes_R N)$.
The advantage to this point of view is that the vanishing of a $\beta_i^F(M,R)$ is closely related to having phantom homology at the $i$th spot of $G_\bullet$, and we can use techniques from \cite{Ab}.

\section{Finite length modules with enough vanishing Frobenius Betti numbers}\label{FBnumber}

Our main goal in this section is to provide an answer to Question~\ref{1st question} under the most general hypotheses on a local ring (e.g., the answer we provide may not be the best possible if we were to assume that the ring is a complete intersection).  In fact, finite projective dimension of a finite length module $M$ over a ring $(R,\mm)$ of dimension $d$  and positive depth is equivalent to having $d+1$ consecutive higher $\beta_i^F(M, R)$'s being zero.  For a full statement, see Corollary~\ref{result}.  This result is a significant generalization of \cite[Proposition 4.1]{SHB}.

\begin{remark}
Let $(G_\bullet, \phi_\bullet)$ be a left complex of finitely generated free modules over a local ring $(R, \mm)$ with all homology of finite length.  Then, on the punctured spectrum, $G_\bullet$ becomes split exact.  Since Frobenius commutes with localization this implies that all homology of $F^e(G_\bullet)$ has finite length.  We will use this fact implicitly whenever we are concerned with resolutions of finite length modules.
\end{remark}

\begin{proposition}\label{d+1}
Let $(R,\mathfrak m)$ be a $d\  (\geq 1)$-dimensional complete local ring of characteristic $p>0.$  Let $(G_\bullet,\phi_\bullet)$ be a complex of finitely generated free $R$-modules and $\phi_{j+1}(G_{j+1})\subseteq\mathfrak m G_{j}$ for all $j\geq 0.$ Suppose that for some $i \geq 1,$ $G_\bullet$ has stably phantom homologies at the $i+j$th spots for all $0\leq j\leq d.$ Then $G_{i+d}=0.$
\end{proposition}
\begin{proof}
 By \cite[Lemma 9.15(b)]{HH}, $F_{R^{red}}^e(G_\bullet\otimes R^{red})$ is phantom at $i+j$th spots for all $0\leq j\leq d$ and $e\geq 0.$ Without loss of generality we may assume that $R$ is reduced and replace the complex $G_\bullet$ by $G_\bullet\otimes R^{red}.$ We use induction on $d.$ We follow the argument given in \cite[Proposition 2.1.7]{Ab}. 
 
 Let $d=1.$ Since $R$ is a complete reduced local ring, we have a test element $c\in R^o$ such that $cH_{i+j}(F_R^e(G_\bullet))=0$ for all $0\leq j\leq 1.$ Choose $t>0$ such that $c^2\notin \mathfrak m^t.$  Let $x\in R^o\cap \mathfrak m^t$ be an element and $S=R/xR.$ For all $e\geq 0,$ consider the short exact sequence of complexes 
$$0\longrightarrow F_R^e(G_\bullet)\overset{\cdot x}\longrightarrow F_R^e(G_\bullet)\longrightarrow F_S^e(G_\bullet\otimes S)\longrightarrow 0.$$
This induces long exact sequence of homology modules
$$\cdots\longrightarrow H_{k+1}(F_R^e(G_\bullet))\longrightarrow H_{k+1}(F_S^e(G_\bullet\otimes S))\longrightarrow H_k(F_R^e(G_\bullet))\longrightarrow\cdots.$$ 
For $k=i,$ we get $c^2H_{i+1}(F_S^e(G_\bullet\otimes S))=0.$ Since $S$ is Artinian, for $e\gg 0,$ the chain maps of $F_S^e(G_\bullet\otimes S)$ are trivial. Hence $c^2(S\otimes G_{i+1})=0$.  The image of $c^2 \ne 0$ in $S$, and the module $S \otimes_R G_{i+1}$ is $S$-free, which implies $G_{i+1}=0$.

Now suppose $d\geq 2$ and $c$ is a test element in $R.$ Then $cH_{i+j}(F_R^e(G_\bullet))=0$ for all $0\leq j\leq d.$ Consider an element $x\in R$ avoiding all minimal primes of $R$ such that $c,x$ is a part of system of parameter of $R.$ Let $S=R/xR.$ Note that $\overline{c}$ (image of $c$ in $S$) is in $S^o.$ Using the long exact sequence of homologies induced by the short exact sequence of complexes 
$$0\longrightarrow F_R^e(G_\bullet)\overset{\cdot x}\longrightarrow F_R^e(G_\bullet)\longrightarrow F_S^e(G_\bullet\otimes S)\longrightarrow 0,$$ we get ${\overline c}^2H_{i+1+j}(F_S^e(G_\bullet\otimes S))=0$ for all $0\leq j\leq d-1.$ By \cite[Lemma 9.15(d)]{HH}, for all $0\leq j\leq d-1$ and $e\geq 0,$ $F_S^e(G_\bullet\otimes S)$ are phantom at $i+1+j$th spots. Then by induction $(G_\bullet\otimes S)_{i+d}=0$ and hence $G_{i+d}=0.$
\end{proof}

As a consequence we generalize the result in \cite[Corollary 4.9]{SHB} to all Noetherian local rings of dimension $d\geq 1.$
\begin{corollary}\label{general}
Let $(R,\mathfrak m)$ be a $d \ (\geq 1)$-dimensional Noetherian local ring of characteristic $p>0$ and let $M$ be an $R$-module of finite length. Suppose, for some $i\geq 1$, that $\beta_{i+j}^F(M,R)=0$ for all $0\leq j\leq d$.  Then $\pd_R M<\infty.$ 

In particular, for any system of parameters $\underline x=x_1,\ldots,x_d,$ if $\beta_{j}^F(R/(\underline x),R)=0$ for all $2\leq j\leq d+1$ then $R$ is Cohen-Macaulay.
\end{corollary}
\begin{proof}
Without loss of generality we may assume that $R$ is complete. Let $G_\bullet$ be a minimal free resolution of $M$ and $R^{eq}=R/I$ where $I$ is the intersection of primary components of the ideal $(0)$ associated to the primes $\mathfrak p$ such that $\dim R/\mathfrak{p}=\dim R.$ 

By \cite[Proposition 2.6]{AL}, $F_{R^{eq}}^e(G_\bullet\otimes R^{eq})$ is phantom at the $i+j$th spots for all $0\leq j\leq d$ and $e\geq 0.$ Then by Proposition \ref{d+1}, we have $(G_\bullet\otimes R^{eq})_{i+d}=0$ and hence $G_{i+d}=0.$

Now suppose that for some system of parameters, $\underline x$, $\beta_{j}^F(R/(\underline x),R)=0$ for all $2\leq j\leq d+1$.  By \cite[Lemma 1.1]{D}, \cite{R}, we have $$\beta_{1}^F(R/(\underline x),R)\leq \lim\limits_{e\to\infty}\frac{\lambda(H_1({\underline x}^{[q]};R))}{q^d}=0$$ where $H_1$ denotes the first Koszul homology. Then by the first part of the corollary and the {\it{new intersection theorem}} \cite{Ro87}, we get $R$ is Cohen-Macaulay.
\end{proof}

The next result follows by the same proof as \cite[Corollary 3.5]{AL}. For the sake of completeness we include the proof.
\begin{corollary}
	Let $(R,\mm)$ be a $d\  (\geq 1)$-dimensional excellent reduced local ring of characteristic $p>0$ and let $M$ be an $R$-module of finite length. Let $T=R^+\mbox{ or }R^{\infty}.$ Suppose $\tor_{i+j}^R(M,T)=0$ for all $0\leq j\leq d$ with some $i\geq 1.$ Then $\pd_R M<\infty.$
	\end{corollary}
\begin{proof}
	Let $(G_\bullet,\alpha_\bullet)$ be a minimal resolution of $M.$ By Proposition \ref{d+1}, it is enough to show that $G_\bullet$ is stably phantom at $i+j$th spots for all $0\leq j\leq d.$ Fix $j\in\{0,\ldots,d\}$ and let $a\in\ker(\alpha_{i+j}^{[q]}).$ Then $a^{1/q}\in \ker(\alpha_{i+j}\otimes_R 1_T)=\im(\alpha_{i+j+1}\otimes_R 1_T).$ Let $a^{1/q} =\sum\limits_{l=1}^nf_{i+j+1}\cdot x_l\in G_{i+j}\otimes_R T$ where each  $x_l\in T$ with $1\leq l\leq n$ and $f_{i+j+1}=\alpha_{i+j+1}\otimes_R1_T$ Let $S = R[x_1,\ldots, x_n]$, which is a module finite extension of $R$. Then $a\in\im(\alpha_{i+j+1}^{[q]}\otimes_R 1_S)\cap G_{i+j}\subseteq \im(\alpha_{i+j+1}^{[q]})_{G_{i+j}}^*.$
	\end{proof}
The next lemma is proved using similar ideas to \cite[Theorem 4.7]{SHB} where the theorem is proved for one-dimensional Noetherian local rings and minimal free resolutions of finite length modules.  Note that if Question~\ref{2nd question} has a positive answer and $i > d+1$, then a module $M$ satisfying the hypotheses of Lemma~\ref{nilpotent} would have to have  $\pd_R M < \infty$.

\begin{lemma}\label{nilpotent}
Let $(R,\mathfrak m)$ be a $d (\geq 1)$-dimensional Noetherian local ring of characteristic $p>0.$ Let $(G_\bullet,\phi_\bullet)$ be a complex of finitely generated free $R$-modules  such that $\phi_{j+1}(G_{j+1})\subseteq\mathfrak m G_{j}$ and $\lambda{(H_{j}(F_R^e(G_\bullet)))}<\infty$ for all $j,e\geq 0.$   Suppose $\im(\phi_{i+1})\subseteq N(R)G_i$ for some $i\geq 1$ where $N(R)$ is the nilradical of $R.$ Then 
\ben
\item $\im(\phi_{i+1})\subset H_{\mathfrak m}^0(G_i).$
\item $H_{i}(F_{R/\mathfrak p}^e(G_\bullet\otimes_R R/\mathfrak p))=0$ for all $e\geq 0$ and for all ${\mathfrak p}\in$ $\spec R\setminus\{\mathfrak m\}.$
\item $\lim\limits_{e\to\infty}\dfrac{\lambda(H_{i}(F_{R}^e(G_\bullet)))}{p^{ed}}=0.$
\een
\end{lemma}
\begin{proof}
Let ${\mathfrak p}\in$ $\spec R\setminus\{\mathfrak m\}.$ Since $\lambda{(H_{j}(F_R^e(G_\bullet)))}<\infty$ for all $j,e\geq 0,$ the following complex 
$$G^q_\bullet\hspace{0.2cm} \cdots\longrightarrow (G_{i+1})_{\mathfrak p}\overset{(\phi_{i+1}^{[q]})_{\mathfrak p}}\longrightarrow (G_{i})_{\mathfrak p}\overset{(\phi_{i}^{[q]})_{\mathfrak p}}\longrightarrow(G_{i-1})_{\mathfrak p}\overset{(\phi_{i-1}^{[q]})_{\mathfrak p}}\longrightarrow\cdots\overset{(\phi_{1}^{[q]})_{\mathfrak p}}\longrightarrow(G_{0})_{\mathfrak p}\longrightarrow 0$$
is split exact. 

By hypothesis, there exits $q_0=p^{e_0}$ such that $\im(\phi_{i+1}^{[q]})=0$ for all $q\geq q_0.$ Hence for any $q\geq q_0,$ $(\phi_{i+1}^{[q]})_{\mathfrak p}=0$ and $(G_{i})_{\mathfrak p}$ splits inside $(G_{i-1})_{\mathfrak p}$ via $(\phi_{i}^{[q]})_{\mathfrak p}.$ Thus $b_i:=\rk((G_{i})_{\mathfrak p})=\rk(G_i)=\rk((\phi_{i}^{[q]})_{\mathfrak p}).$ Note that $G_i$ and $G_{i-1}$ are free modules and localizing and taking powers can only decrease the rank of $\phi_i.$ Thus we have $b_i=\rk(\phi_{i}^{[q]})$ for all $q\geq 1.$  

For all $q\geq q_0,$ consider the complex $$L_\bullet^q: \hspace{0.2cm} 0\longrightarrow G_i\overset{\phi_{i}^{[q]}}\longrightarrow G_{i-1}\longrightarrow 0.$$ Since $L_\bullet^q\otimes R_{\mathfrak p}$ is split acyclic, by \cite[Proposition 1.4.12((a)$\Rightarrow$(b))]{BH}, $I_{b_i}(\phi_{i}^{[q]})\nsubseteq {\mathfrak p}.$ Therefore $I_{b_i}(\phi_{i}^{[1]})\nsubseteq {\mathfrak p}.$ 

$(1)$ Consider the complex $$L_\bullet^1: \hspace{0.2cm} 0\longrightarrow G_i\overset{\phi_{i}^{[1]}}\longrightarrow G_{i-1}\longrightarrow 0.$$ Let ${\mathfrak p}\in$ $\spec R\setminus\{\mathfrak m\}.$ By \cite[Proposition 1.4.12((b)$\Rightarrow$(a))]{BH}, we have $L_\bullet^1\otimes R_{\mathfrak p}$ is split acyclic and hence $\ker((\phi_{i}^{[1]})_{\mathfrak p})=0.$ Then $$(\im(\phi_{i+1}^{[1]}))_{\mathfrak p}=\im((\phi_{i+1}^{[1]})_{\mathfrak p})\subseteq \ker((\phi_{i}^{[1]})_{\mathfrak p})=0.$$ Hence if ${\mathfrak p}\in$ Spec$R\setminus\{\mathfrak m\}$ then ${\mathfrak p}\notin$ Supp$(\im(\phi_{i+1})).$ Thus $\im(\phi_{i+1})\subset H_{\mathfrak m}^0(G_i).$

$(2)$ Let ${\mathfrak p}\in$ Spec$R\setminus\{\mathfrak m\}.$  For all $q\geq 1,$ consider the complex of finite free $R/\mathfrak p$-modules. $$T_\bullet^q:  \hspace{0.2cm} 0\longrightarrow G_i\otimes_R R/\mathfrak p\xrightarrow{\phi_{i}^{[q]}\otimes_R 1_{R/\mathfrak p}} G_{i-1}\otimes_R R/\mathfrak p\longrightarrow 0.$$ Since  for all $q\geq 1,$ $I_{b_i}(\phi_{i}^{[q]})\nsubseteq {\mathfrak p},$ we have $\grade (I_{b_i}(\phi_{i}^{[q]}\otimes_R 1_{R/\mathfrak p}))\geq 1$  for all $q\geq 1.$ Hence by the Buchsbaum-Eisenbud Theorem \cite{BE}, \cite[Theorem 1.4.13]{BH},
 we have $T_\bullet^q$ is acyclic for all $q\geq 1$.  Therefore $\phi_{i}^{[q]}\otimes_R 1_{R/\mathfrak p}$ is injective  for all $q\geq 1.$  Since $\im(\phi_{i+1})\subseteq \mathfrak pG_i,$ we have $H_{i}(F_{R/\mathfrak p}^e(G_\bullet\otimes_R R/\mathfrak p))=0$ for all $e\geq 0.$

$(3)$ Consider a filtration $0=M_0\subseteq M_1\subseteq\cdots\subseteq M_l=R$ of $R$ such that $M_j/M_{j-1}\cong R/\mathfrak p_j$ for some $\mathfrak p_j\in \spec R.$  Then by \cite[Proposition 1 (a),(b)]{S} and part $(2)$ of the Lemma, we have $\lim\limits_{e\to\infty}\dfrac{\lambda(H_{i}(F_{R}^e(G_\bullet)))}{p^{ed}}=\sum\limits_{j=1}^l\lim\limits_{e\to\infty}\dfrac{\lambda(H_{i}(F_{R/\mathfrak p_j}^e(G_\bullet\otimes_R R/\mathfrak p_j)))}{p^{ed}}=0.$ 
\end{proof}

Using the same proofs of \cite[Lemma 4.6]{SHB} and \cite[Proposition 1]{S}, we get the following lemma:
\begin{lemma}\label{primes}
Let $(R,\mathfrak m)$ be a one-dimensional Noetherian local ring of characteristic $p>0.$ Let $(G_\bullet,\phi_\bullet)$ be a complex of finitely generated free $R$-modules  such that $\phi_{j+1}(G_{j+1})\subseteq\mathfrak m G_{j}$ and $\lambda{(H_{j}(F_R^e(G_\bullet)))}<\infty$ for all $j,e\geq 0.$ Suppose $\im(\phi_{i+1})\nsubseteq {\mathfrak p}G_i$ for some ${\mathfrak p}\in$ $\min(R)$ and $i\geq 1.$ Then $$\lim\limits_{e\to\infty}\frac{\lambda{(H_{i}(F_R^e(G_\bullet)))}}{q}>0.$$
\end{lemma}

\begin{theorem}\label{important}
	Let $(R,\mathfrak m)$ be a $d \ (\geq 1)$-dimensional formally equidimensional local ring of characteristic $p>0.$ Let $(G_\bullet,\phi_\bullet)$ be a complex of finitely generated free $R$-modules such that $\phi_{j+1}(G_{j+1})\subseteq\mathfrak m G_{j}$ and $\lambda{(H_{j}(F_R^e(G_\bullet)))}<\infty$ for all $j,e\geq 0.$ Suppose $$\lim\limits_{e\to\infty}\frac{\lambda{(H_{i+j}(F_{R}^e(G_\bullet)))}}{q^d}=0.$$ for all $0\leq j\leq d-1$ with some $i> 0.$ Then $\im({\phi_{i+d}})\subset N(R)G_{i+d-1}$ where $N(R)$ is the nilradical of $R.$
\end{theorem}
\begin{proof}
	Without loss of generality we may assume that $R$ is complete and equidimensional. We use induction on the dimension of $R.$ 
	
	If $d=1$, then the result holds by Lemma \ref{primes}.
	
	 Let $d\geq 2.$ By \cite[Proposition 1]{S}, for all $\mathfrak p\in \min(R),$ we have $$\lim\limits_{e\to\infty}\frac{\lambda{(H_{i+j}(F_{R/\mathfrak p}^e(G_\bullet\otimes_R R/\mathfrak p)))}}{q^d}=0.$$ for all $0\leq j\leq d-1.$ Suppose we can prove the result for the complete local rings $R/\mathfrak p$ for all  $\mathfrak p\in \min(R).$ Then $\im({\phi_{i+d}}\otimes_R R/\mathfrak p)\subset N(R/\mathfrak p)(G_{i+d-1}\otimes_R R/\mathfrak p)=0\mbox{ in each }R/\mathfrak p.$ Therefore $\im(\phi_{i+d})\subset N(R)(G_{i+d-1})$, as desired.
	
	Replacing $R$ by $R/\mathfrak p$ for each $\mathfrak p\in \min(R),$ we may assume that $R$ is  a complete local domain and $$\lim\limits_{e\to\infty}\frac{\lambda{(H_{i+j}(F_{R}^e(G_\bullet)))}}{q^d}=0.$$ for all $0\leq j\leq d-1.$  We will show $\phi_{i+d}=0.$
	
	By \cite[Proposition 2.6]{AL}, $F_R^e(G_\bullet)$ is phantom at $i+j$th spot for all $0\leq j\leq d-1$ and $e\geq 0.$ Let $c$ be a test element in $R.$ Choose a nonzero element $x\in \mathfrak m\setminus A$ where  $A=\min(R/(c)).$ Then $c,x^n$ is a part system of parameter of $R$ for all $n\geq 1.$ Define  $S_n=R/x^nR$ for all $n\geq 1$ and let $T = (S_n)^{red} = R/\sqrt{x^nR} = R/\sqrt{xR}$.  Note that $\dim T = d-1$ and $T$ is equidimensional.
	
	For all $n\geq 1,$ consider the short exact sequences of complexes
	$$0\longrightarrow F_{R}^e(G_\bullet)\overset{. {x^n}}\longrightarrow F_{R}^e(G_\bullet)\longrightarrow F_{S_n}^e(G_\bullet\otimes_R S_n)\longrightarrow 0.$$
	
	By \cite[Lemma 9.15(d)]{HH}, $cH_{i+j}(F_R^e(G_\bullet))=0$ for all $0\leq j\leq d-1$ and $e\geq 0.$ Therefore for all $0\leq j\leq d-1,$ $e\geq 0$ and $n\geq 1,$ $c^{2}H_{i+j}(F_{S_n}^e(G_\bullet\otimes_{R} S_n))=0.$ Note that $\overline{c^{2}}$ (image of $c^{2}$ in $S_n$) is in $(S_n)^o.$ Hence, by \cite[Lemma 9.15(d)]{HH}, for all $1\leq j\leq d-1,$ $e\geq 0$ and $n\geq 1,$ $F_{S_n}^e(G_\bullet\otimes_{R} S_n)$ is phantom at $i+j$th spot. 
	
	Since $T=(S_n)^{red}$ is formally equidimensional, we have $T=T^{eq}$ and by  \cite[Lemma 9.15(b)]{HH} and \cite[Proposition 2.6]{AL}, for all $1\leq j\leq d-1,$ and $e\geq 0$,  $F_{T^{eq}}^e(G_\bullet\otimes_R T^{eq})$ is phantom at the $i+j$th spot. Hence by \cite[Proposition 2.6]{AL}, for all $1\leq j\leq d-1$  
	$$\lim\limits_{e\to\infty}\frac{\lambda{(H_{i+j}(F_{T}^e(G_\bullet\otimes_R T)))}}{q^{d-1}}=0.$$
	
	By induction and the fact that $T$ is reduced, we have 
	$$\im((\phi_{i+d}\otimes_R S_n) \otimes_{S_n}T)\subseteq N(T)((G_{i+d-1}\otimes_R S_n) \otimes_{S_n}T)=0\mbox{ in }T$$ where  $N(T) =0$ is the nilradical of $T.$
	Therefore $$\im(\phi_{i+d}\otimes_R S_n) \subseteq N(S_n)(G_{i+d-1}\otimes_R S_n) \mbox{ where } N(S_n)\mbox{ is the nilradical of } S_n.$$ 
	
	Hence for all integers $n\geq 1,$ by Lemma \ref{nilpotent}, $\im(\phi_{i+d}\otimes_R S_n)\subseteq H_{\mathfrak mS_n}^0(G_{i+d-1}\otimes_R S_n).$ 
 Let $a\in \im \phi_{i+d}.$ Then for all integers $n\geq 1,$ $a-b_n\in(x^{n})$ for some $b_n\in (x^{n}):_{G_{i+d-1}}\mathfrak m^{\infty}.$ Then,
 using colon-capturing \cite[Theorem 7.9]{HH} for the last equality,
  $$
 b_n\in (x^{n}):_{G_{i+d-1}}\mathfrak (c)^{\infty}= ((x^{n}):_R\mathfrak (c)^{\infty})^l=(x^{n})^* R^l
 $$
  where $l$ is the rank of $G_{i+d-1}.$ Since $c$ is a test element of $R,$ for all integers $n\geq 1,$ we have $ca\in (x^n).$ Therefore $ca\in\bigcap\limits_{n\geq 1}{(x^{n})}=0.$ Since $c$ is a nonzerodivisor of $R,$ we get $a=0.$
\end{proof}

\begin{corollary}\label{result}
	Let $(R,\mathfrak m)$ be a $d\ (\geq 1)$-dimensional Noetherian local ring of characteristic $p>0.$  Let $M$ be an $R$-module of finite length. Then the following are equivalent.
	\begin{itemize}
		\item [$(1)$] $\beta_{i}^F(M,R)=0$ for all $i\geq 1.$
		\item [$(2)$] For some $i \ge 0$, $\beta^F_{i+j}(M,R)=0$ for all $0\leq j\leq d$.
		\item [$(3)$] $\pd_R M<\infty.$ 
		\item [$(4)$] $R$ is Cohen-Macaulay and for some $i \ge 1$, $\beta^F_{i+j}(M,R)=0$ for all $0\leq j\leq d-1$.
		\item [$(5)$] $(R,\mathfrak m)$ is formally equidimensional, $\depth R\geq 1$ and for some $i \ge 1$, $\beta_{i+j}^F(M,R)=0$ for all $0\leq j\leq d-1$.
	\end{itemize}
\end{corollary}
\begin{proof}
The equivalence $(1)\Rightarrow (2)\Rightarrow (3)\Rightarrow (1)$ follows from Corollary \ref{general} and \cite[Theorem 1.7]{PS}. It is clear that $(4)\Rightarrow (5).$
\\	$(5)\Rightarrow (3)$ Let $(G_\bullet,\phi_\bullet)$ be a minimal resolution of $M.$ By Theorem \ref{important}, $\im({\phi_{i+d}})\subset N(R)G_{i+d-1}.$ Therefore by Lemma \ref{nilpotent}, $$\im({\phi_{i+d}})\subset H_{\mathfrak m}^0(G_{i+d-1})=0.$$
$(3)\Rightarrow (4)$ Since $\lm(M)<\infty,$ by the {\it improved new intersection theorem} \cite[Corollary 9.4.2]{BH} and the Auslander-Buchsbaum Theorem \cite[Theorem 1.33]{BH}, we have $R$ is Cohen-Macaulay.
\end{proof}
Using  ideas similar to   \cite[Proposition 4.11] {SHB} and Corollary \ref{result}, we generalize the result  \cite[Proposition 4.11]{SHB} for all Noetherian local rings of dimension $d\geq 1.$ 
\begin{proposition}
	Let $(R,\mathfrak m,k)$ be a $d (\geq 1)$-dimensional formally equidimensional local ring of characteristic $p>0.$ Let $I$ be an $\mm$-primary integrally closed ideal. If $\beta^F_{i+j}(R/I,R)=0$ for all $0\leq j\leq d-1$ with some $i\geq 1$ then $R$ is regular.
	\end{proposition}
\begin{proof}
	Let $G_\bullet$ be a minimal free resolution of $R/I$ and ${\mathfrak p}\in$ Spec$R\setminus\{\mathfrak m\}.$ By Theorem \ref{important} and Lemma \ref{nilpotent}, we have $\tor_{i+d-1}^R(R/I,R/\mathfrak p)=0.$ Therefore by \cite[Corollary 3.3]{CHKV}, pd$_RR/\mathfrak p<\infty$ and thus by the Auslander-Buchsbaum Theorem \cite[Theorem 1.33]{BH}, $\depth R\geq 1.$ Hence by Corollary \ref{result}, $\tor_{j}^R(R/I,k)=0$ for some $j \gg 0.$ Again using \cite[Corollary 3.3]{CHKV}, we get pd$_R k<\infty.$ Hence $R$ is regular.
	\end{proof}
\begin{question}
	Let $(R,\mathfrak m,k)$ be a $d (\geq 2)$-dimensional Cohen-Macaulay local ringof characteristic $p>0,$ $M$ be an $R$-module of finite length and $\alpha$ be an integer with $1\leq \alpha\leq d-1$. What conditions on $R$ allow us to conclude that vanishing of $\beta^F_{i+j}(M,R)$ with some $i>0$ and all $0\leq j\leq \alpha-1$ force $\pd_RM<\infty?$
	
	\end{question}
	
\section{Syzygies with finite length}	

In this section we turn from dealing with asymptotic measures of length via Frobenius, to the open questions regarding syzygies of modules of finite length.  In particular we want to address the second question asked by De Stefani, Huneke, and N\'{u}\~{n}ez-Betancourt \cite{SHB}:
\begin{question}[Question 1.2]
Let $R$ be a $d$-dimensional local ring, and let $M$ be a finitely generated $R$-module such that $\pd_R(M) = \infty$ and $\lambda(M) < \infty$.  If $i > d+1$, then must the length of the $i$th syzygy be infinite?
\end{question}

One way that they gain information concerning this question is to show that in dimension one, a finite length syzygy of a finite length module $M$ has a length that can be computed as an alternating sum of lengths of  Tor's of $M$ against  $R/xR$ where $x$ is in a suitably high power of $\mm$ (\cite[Proposition 5.9]{SHB}).  We give a simpler proof of this fact when $R$ has any dimension and we can replace $xR$ by an $\mm$-primary ideal $J$ in a suitable high power of $\mm$ --- see Proposition~\ref{extended}.  From this result we can, in dimension two, derive Theorem~\ref{dimension two}, that for a sufficiently general $x \in \mm$, we may use the alternating sum of Tor's of $M$ against $R/xR$ (the proof must be done indirectly, because $\lambda(R/xR) = \infty)$, and then a similar argument as given in \cite{SHB} shows that no third syzygy of a finite length module can have finite length.  Perhaps it is the case that Question~\ref{2nd question} should be strengthened to ask whether or not, if $\dim R > 1$, {\it any} higher syzygy of a finite length module can be finite length!

\begin{remark}\label{completion}
	Let $(R,\mm)$ be a Noetherian local ring.  $\hat R$ denotes the $\mm$-adic completion of $R.$ Then $\syz_n^{\hat R}(M\otimes_R {\hat R})={\hat R}\otimes_R \syz_n^{ R}(M).$ 
	\end{remark}
\begin{remark}
	If $(R,\mathfrak m)$ is a Noetherian local ring with $\depth R\geq 1$ and $M$ is an $R$-module of finite length with infinite projective dimension then $\lambda(\syz_iM)=\infty$ for all $i>0$ \cite[Lemma 3.4]{A}. 
	\end{remark}
	
We start by generalizing the known result that Question~\ref{2nd question} has a positive answer when $(R,\mm)$ is Buchsbaum.  We show below that, for rings of positive dimension, if the size of the socle of $R$ is large compared to the size of 	$H^0_\mm(R)$, then Question~\ref{2nd question} has an affirmative answer.

\begin{theorem}\label{big socle}
	Let $(R,\mathfrak m, k)$ be a Noetherian local ring  of dimension $d\geq 1$ and $\depth R=0$.  Set $t = \lambda(H_{\mm}^0(R))-\lambda(\soc(R))$ and $l = \dim(\soc(R))$.  Let $M$ be an $R$-module of finite length. 
	\begin{enumerate}
	\item If $d =1$ and $l > t$ then $\lambda(\syz_iM)=\infty$ for all $i>0.$ 
	\item If $d\geq 2$  and $l \ge t$ then $\lambda(\syz_iM)=\infty$ for all $i>0.$ 
		\end{enumerate}
	\end{theorem}
\begin{proof}
In the case that $t=0$ (e.g., $R$ is Buchsbaum), the result is already known.  (See \cite[Proposition 5.3]{SHB}, \cite[Proposition 4.4]{A}.   The Buchsbaum hypothesis in the statement is not really used, merely the condition that $H^0_\mm(R)$ is a vector space.)  Also, the proof works for arbitrary positive dimension.

By \cite[Lemma 4.2]{A}, $\lambda(\syz_iM)=\infty$ for all $1\leq i\leq d.$ 

 Suppose $\lambda(\syz_{j+1}M)<\infty$ for some $j\geq d.$ Then by \cite[Lemma 5.2]{SHB}, $\tor_j^R(M,R/H_{\mm}^0(R))=0.$ 
Denote $H_{\mm}^0(R)$ and $\soc(R)$ by $I$ and $J$ respectively. Since $\depth R=0,$ we get $J\neq 0.$ Consider the short exact sequences of $R$-modules
	$$0\longrightarrow I/J\longrightarrow R/J\longrightarrow R/I\longrightarrow0$$ and $$0\longrightarrow J\longrightarrow R\longrightarrow R/J\longrightarrow0.$$ 
	From the long exact sequences of homologies induced by the above two short exact sequences, we get that 
	 $$\tor_j^R(M,I/J)\longrightarrow \tor_j^R(M,R/J)\longrightarrow 0$$ and  $$\tor_{j}^R(M,R/J)\cong \tor_{j-1}^R(M,J)={\tor_{j-1}^R(M,k)}^{{\oplus} l}.$$ Therefore $\lambda(\tor_j^R(M,I/J))\geq l\beta_{j-1}(M).$

For any finite length module $N$ and any index $i$, it is clear by induction on the length of $N$ that $\lambda(\tor_i(M,N)) \le \lambda(N) \lambda(\tor_i(M,k))$. Hence
$$l\beta_{j-1}(M)\leq\lambda(\tor_j^R(M,I/J))\leq t\lambda(\tor_j^R(M,k))=t\beta_{j}(M).$$

$(1)$ If $l>t,$ then $\beta_j>\beta_{j-1}$ which contradicts $\lambda(\syz_{j+1}M)<\infty.$

$(2)$ Since $l\geq t,$ by \cite[Proposition 5.5]{SHB}, we get $d=1$, which gives a contradiction.
	\end{proof}

The proposition below is well known.  We provide a proof for convenience of the reader.
	\begin{proposition}\label{length}
	Let $(R,\mathfrak m)$ be a Noetherian local ring of dimension $d\geq 1.$ Let $$F. \hspace{2mm} 0\longrightarrow M_n\overset{\delta_n}\longrightarrow M_{n-1}\overset{\delta_{n-1}}\longrightarrow\cdots\longrightarrow M_1\overset{\delta_1}\longrightarrow M_0\overset{\delta_0}\longrightarrow 0$$ be a complex of finitely generated $R$-modules of finite length. Then $$\sum\limits_{i=0}^n(-1)^i\lambda(M_i)=\sum\limits_{i=0}^n(-1)^i\lambda(H_i(F.)).$$
		\end{proposition}
	\begin{proof}
We prove using induction on $n.$ Let $n=1.$ Then we have the following two exact sequences of $R$-modules
$$0\longrightarrow\im(\delta_1) \longrightarrow M_0\longrightarrow H_0(F.)\longrightarrow 0,$$
$$0\longrightarrow H_1(F.)\longrightarrow M_1\longrightarrow \im(\delta_1)\longrightarrow 0.$$
Hence we get $\lambda(M_0)-\lambda(M_1)=\lambda(H_0(F.))+\lambda(\im (\delta_1))-\lambda(M_1)=\lambda(H_0(F.))-\lambda(H_1(F.)).$

Now suppose $n\geq 2$ and the result holds for all $m\leq n-1.$ Let $G_\bullet$ be the following complex $$0\longrightarrow M_{n-1}\overset{\delta_{n-1}}\longrightarrow M_{n-2}\overset{\delta_{n-2}}\longrightarrow\cdots\longrightarrow M_1\overset{\delta_1}\longrightarrow M_0\overset{\delta_0}\longrightarrow 0.$$ By the induction hypothesis we get $$\sum\limits_{i=0}^{n-1}(-1)^i\lambda(M_i)=\sum\limits_{i=0}^{n-1}(-1)^i\lambda(H_i(G_\bullet)).$$
From the following two exact sequences of $R$-modules
$$0\longrightarrow\im(\delta_n) \longrightarrow \ker(\delta_{n-1}) \longrightarrow H_{n-1}(F.)\longrightarrow 0,$$
$$0\longrightarrow H_n(F.)\longrightarrow M_n\longrightarrow \im(\delta_n)\longrightarrow 0,$$
 we get  $\lambda(\im(\delta_n))=\lambda(\ker(\delta_{n-1}))-\lambda(H_{n-1}(F.))$ and $$\lambda(M_n)=\lambda(H_n(F.))+\lambda(\im(\delta_n))=\lambda(H_n(F.))+\lambda(\ker(\delta_{n-1}))-\lambda(H_{n-1}(F.)).$$
Note that $H_{i}(G_\bullet)=H_{i}(F.)$ for all $0\leq i\leq n-2$ and  $H_{n-1}(G_\bullet)=\ker(\delta_{n-1}).$ Therefore we have 
\beqn
\sum\limits_{i=0}^n(-1)^i\lambda(M_i)&=&\sum\limits_{i=0}^{n-1}(-1)^i\lambda(M_i)+(-1)^{n}\lambda(M_{n})\\&=&\sum\limits_{i=0}^{n-1}(-1)^i\lambda(H_i(G_\bullet))+(-1)^{n}\lambda(M_{n})\\&=&\sum\limits_{i=0}^{n-2}(-1)^i\lambda(H_i(G_\bullet))+(-1)^{n-1}\ker(\delta_{n-1})+(-1)^{n}\lambda(M_{n})\\&=& \sum\limits_{i=0}^n(-1)^i\lambda(H_i(F.)).
\eeqn
\end{proof}

We are now ready to extend \cite[Proposition 5.9]{SHB} to a large class of ideals in  local rings of arbitrary dimension.  Since the result is defined in terms of an alternating sum of lengths of Tors, we make the following definition:
\begin{definition} Let $(R,\mm)$ be local.  For any finite length $R$-module $M$ and any finitely generated module $N$, let $\sigma_i(M,N) := \sum_{j=0}^i (-1)^{i-j+1} \lambda(\tor_j(M,N))$.
\end{definition}

\begin{remark}\label{addsigma}  In the case that $0 \to N_1 \to N_2 \to N_3 \to 0$ is a short exact sequence, $M$ has finite length, and $\tor_{i+1}(M, N_3) = 0$ (or, more generally when the connecting map $\tor_{i+1}(M, N_3) \to \tor_i(M, N_1)$ is zero), then the short exact sequence in Tor's and additivity of lengths shows that $\sigma_i(M,N_2) = \sigma_i(M, N_1) + \sigma_i(M, N_3)$.
\end{remark}

Proposition~\ref{extended} shows that under the hypothesis that the $i+1$st syzygy of a finite length module $M$ has finite length, the length of the syzygy can be expressed as $\sigma_i(M, R/J)$ for any $\mm$-primary ideal in a suitably high power of $\mm$. 

\begin{proposition}\label{extended}
	Let $(R,\mathfrak m)$ be a Noetherian local ring of dimension $d\geq 1$ and $M$ be an $R$-module of finite length.  Suppose for some $i\geq 1,$ $0<\lambda(\syz_{i+1}M)<\infty.$ The following are true.
\ben
\item There exists an integer $b>0$ such that for any ideal $J\subset\mm^b,$ $\tor_{i+1}^R(M,R/J)=0.$
\item There exists an integer $n\ge b$ (where $b$ is as in part (1)) such that all  $\mm$-primary ideals $J$ contained in $\mm^n$  satisfy the following
$$J\subset \ann(M)\cap\ann(\syz_{i+1}M) \mbox{ and\ } \lambda(\syz_{i+1}M)= \sigma_i(M, R/J).$$
\item There exists a parameter ideal $J=(x_1,\ldots,x_d)$ such that $(0:x_i)\cong H_{\mm}^0(R)$ for all $i=1,\ldots d$,  and $J$ satisfies the properties of $(1)$ and $(2)$.
 \een
	\end{proposition}
 \begin{proof}
 Let $b_1$ be an integer such that $ H_{\mm}^0(R)=(0:\mm^{b_1})$.  Let $k_1$ be the Artin-Rees number for $H^0_\mm(R) \subseteq R$ with respect to $\mm$, i.e., for all $m \ge k_1$, $\mm^m \cap H^0_\mm(R) \subseteq \mm^{m-k_1} H^0_\mm(R)$.  Set $b = b_1 + k_1$.
 
 	$(1)$ Let $(G_\bullet,\delta_\bullet)$ be a minimal free resolution of $M.$ Consider the complex $G_\bullet\otimes_R R/J$ where $J$ is any ideal of $R$ such that $J\subset\mm^b.$  Let $z\in G_{i+1}$ be such that $\overline z\in \ker (\delta_{i+1}\otimes 1_{R/J})$ (here $``-"$ denotes the image in $G_{i+1}/JG_{i+1}$). Since $\lambda(\im (\delta_{i+1}))<\infty,$ we have  $\delta_{i+1}(z)\in JG_i\cap H_{\mm}^0(G_i) \subseteq \mm^{b-k_1}H^0_\mm(R)=0.$ Therefore $\overline z\in \im(\delta_{i+2}\otimes 1_{R/J})$, showing that $\tor_{i+1}(R/J, M) = 0$.

$(2)$ 	Let $(G_\bullet,\delta_\bullet)$ be a minimal free resolution of $M.$	Note that $$(*)\hspace{2mm}0\longrightarrow \syz_{i+1}M\overset{\iota}\longrightarrow G_{i}\overset{\delta_{i}}\longrightarrow\cdots\longrightarrow G_1\overset{\delta_1}\longrightarrow G_0\overset{\delta_0}\longrightarrow 0$$ is an acyclic complex of finitely generated $R$-modules. Let $G_j=R^{\beta_j (M)}$ for all $0\leq j\leq i.$ Since $\lambda(M)<\infty$ and $\lambda(\syz_{i+1}(M))<\infty,$ we have $\sum\limits_{j=0}^i(-1)^{j}\beta_j(M)=0.$ 

Let $t>0$ be an integer such that $\mm^t\subseteq \ann_R(M)\cap\ann_R(\syz_{i+1}M).$ By the Artin-Rees Lemma, there exists an integer $k>0$ such that for all $l\geq k,$ we have $$\mm^lG_i\cap\syz_{i+1}M\subseteq \mm^{l-k}\syz_{i+1}M.$$ Let $n=k+t+b.$ Let $J$ be an $\mm$-primary ideal such that $J\subseteq \mm^n.$ Then  we have $$JG_i\cap\syz_{i+1}M\subseteq \mm^nG_i\cap\syz_{i+1}M\subseteq\mm^{n-k}\syz_{i+1}M =0.$$

Tensoring $(*)$ with $R/J$ we get the following  complex of finitely generated $R$-modules of finite length,
	$$G_\bullet': \hspace{2mm} 0\longrightarrow \syz_{i+1}M\overset{\iota\otimes 1_{R/J}}\longrightarrow G_{i}/JG_{i}\overset{\delta_{i}\otimes 1_{R/J}}\longrightarrow\cdots\longrightarrow G_1/JG_1\overset{\delta_1\otimes 1_{R/J}}\longrightarrow G_0/JG_0\overset{\delta_0\otimes 1_{R/J}}\longrightarrow 0.$$ 
	Note that
$\ker(\iota\otimes R/J)=JG_{i}\cap \syz_{i+1}M=0.$
Applying Proposition \ref{length} to the complex $G_\bullet' $, we get $$(-1)^{i+1}\lambda(\syz_{i+1}M)=\sum\limits_{j=0}^{i}(-1)^{j}\lambda(H_j(G_\bullet'))+(-1)^{i+1}\lambda(\ker(\iota\otimes R/J))=\sum\limits_{j=0}^{i}(-1)^{j}\lambda(H_j(G_\bullet')).$$

$(3)$  Let $X = \ass(R) - \{\mm\}$.  If we pick elements $\underline x=x_1,\ldots,x_d$ of $\mm^{n}$ (where $n \ge b_1$ is as in part $(2)$) for our system of parameters  in the usual way, while also avoiding the finite set $X$ at each stage, the resulting system of parameter satisfies the properties of $(1)$ and $(2)$.  Moreover, $x_iH^0_\mm(R) \subseteq \mm^{b_1}H^0_\mm(R) = 0$ and $x_i$ avoids all non-$\mm$-primary components of $0$, so $(0:x_i) = H^0_\mm(R)$.  Define $J=(\underline x).$ 
\end{proof}

\begin{lemma}\label{add}
Let $(R,\mathfrak m)$ be a Noetherian local ring  and  let $M$ be an $R$-module of finite length. Let $y\in R$ be an element such that $(0:(y))\cong H_{\mm}^0(R).$  Then
\ben 
\item  $\lm(\tor_1^R(M,R/(y)))=\lm(M\otimes (y))=\lm(M/H_{\mm}^0(R)M)$ and for all $j\geq 2,$  $\lm(\tor_j^R(M,R/(y)))=\lm(\tor_{j-1}^R(M,(y)))=\lm(\tor_{j-1}^R(M,R/H_{\mm}^0(R))).$ 
\item $\lm(\tor_1^R(M,R/H_{\mm}^0(R)))=\lm(M\otimes H_{\mm}^0(R))-\lambda(M)+\lambda(M/H_{\mm}^0(R)M)$ and for all $j\geq 2,$ \\$\lm(\tor_j^R(M,R/H_{\mm}^0(R)))=\lm(\tor_{j-1}^R(M,H_{\mm}^0(R))).$ 
\een
\end{lemma}	
\begin{proof}
$(1)$ Since $(y)\cong R/H_{\mm}^0(R),$ the result follows from the long exact sequence of homology modules induced by the following short exact sequence tensored with $\otimes_R M,$ $$0\longrightarrow (y)\longrightarrow R \longrightarrow R/(y)\longrightarrow 0.$$

$(2)$ The equalities follow from the long exact sequence of homology modules induced by the short exact sequence below tensored with $\otimes_R M,$ $$0\longrightarrow H_{\mm}^0(R) \longrightarrow R \longrightarrow R/H_{\mm}^0(R)\longrightarrow 0.$$ 
\end{proof}
	
 When $(R,\mm)$ is a $d$-dimensional ring and $M$ has finite length, then   $\sigma_i(M, R/J)$ makes sense for {\it any} ideal $J$, and we get a complex of  modules $G'_\bullet$ as in the proof of Proposition~\ref{extended}.  
 When $J$ is $\mm$-primary,   we can compute using Proposition~\ref{length}. 
 However, when $J$ is not $\mm$-primary, the complex $G'_\bullet$ has finite length homology but is not itself made up of finite length modules so 
 we cannot {\it compute} $\sigma_i(M, R/J)$ using the help of Proposition~\ref{length}.  Below we provide an interesting observation in the case of two-dimensional rings. We show that we can compute the length of a finite length $i+1$st syzygy of $M$ using $\sigma_i(M, R/x_2R)$ when $x_2$ is sufficiently general and in a high power of $\mm$, thus, as in the proof given in \cite{SHB} in dimension one, we can rule out third syzgyies of finite length in dimension two.  In a sense, we are suggesting that with regard to Question~\ref{2nd question}, dimension $d=1$ is special in allowing a $d+1$st syzygy to be finite length.  For dimension $d \ge 2$, perhaps the right question to ask is if {\it any} higher syzygy can be finite length?

\begin{theorem}\label{dimension two}
	 Let $(R, \mm)$ be a $2$-dimensional Noetherian local ring and let $M$ be an $R$-module of finite length.  Suppose $\syz_{i+1}M\neq 0$ for some $i\geq 0.$ Then the following are true.
\ben
\item Suppose $R$ is complete and $\lm(\syz_{i+1}M)<\infty$.  Let $n$ be as in Proposition~\ref{extended}(2).  Then there exists a system of parameters $x_1,x_2\in\mm^n$  such that $x_1$ is a  Cohen-Macaulay multiplier for $R,$ $(0:(x_2))=H^0_\mm(R)$ and
$\lm(\syz_{i+1}M)=\sigma_i(M, R/x_2R).$
\item Suppose $0\leq i\leq 2,$ then $\lm(\syz_{i+1}M)=\infty.$
\een		 
\end{theorem}

\begin{proof}
$(1)$ Let $\mathfrak a=\mathfrak a_0\mathfrak a_1$ where $\mathfrak a_i=\ann H_{\mm}^i(R)$ for $i=0,1.$	By \cite[Theorem 8.1.1]{BH}, $\dim  R/\mathfrak a\leq 1.$ Hence using \cite[Corollary 8.1.3]{BH}, we choose a Cohen-Macaulay multiplier (for the definition see \cite{HH92}) for $R,$ $y_1\in \mathfrak a\setminus A$ where $A=\bigcup\{\mathfrak p\in \spec R:\dim R/\mathfrak p=\dim R\}.$ Choose $y_2\in\mm^n\setminus B\cup C$, $B=\bigcup\{\mathfrak p\in \ass R:\dim R/\mathfrak p>0\}$ and $C=\bigcup\min (R/(y_1))$ (recall $n$ is such that all  $\mm$-primary ideals $J$ contained in $\mm^n$  satisfy the following
	$J\subset \ann(M)\cap\ann(\syz_{i+1}M) \mbox{ and\ } \lambda(\syz_{i+1}M)= \sigma_i(M, R/J).$) Then $y_1^n,y_2$  is a system of parameters for $R$.  Let $t\gg0$ such that $(0:(y_1^n)^{\infty})=(0:y_1^{nt}).$ For some power $m$, $(y_1^{nt}:y_2^\infty) = (y_1^{nt}: y_2^m)$.  Let $ x_1=y_1^{nt}$ and $x_2 = y_2^m$. Then  $(x_1:x_2) = (x_1:x_2^\infty),$ $(0:x_1^{\infty})=(0:x_1)$ and By Proposition~\ref{extended}(3), $(0:x_2) = (0:x_2^\infty) (= H^0_\mm(R)).$  Therefore $x_1,x_2$  is a system of parameters for $R$ and  by \cite[Corollary 8.1.4]{BH}, $x_1$ is a Cohen-Macaulay multiplier.

We consider the short exact sequence
\begin{equation}\label{keySES}
0 \to \dfrac R {x_2((x_1(x_1:x_2)):x_2)} \to \dfrac R {x_1(x_1:x_2)} \oplus \dfrac R {x_2R} \to \dfrac R {x_2R + x_1(x_1:x_2)} \to 0.
\end{equation}
By Proposition~\ref{extended}(1), $\tor_{i+1}(M,R/(x_2, x_1(x_1:x_2))=0$ and by Remark~\ref{addsigma}, $\sigma_i(M,-)$ is additive on the above sequence. Hence by Proposition~\ref{extended}(2), $\lambda(\syz_{i+1}M) = \sigma_i(M, R/(x_2, x_1(x_1:x_2))).$ 

Observe that $x_2((x_1(x_1:x_2)):x_2) = x_2x_1(x_1:x_2)$.  The containment $\supseteq$ is clear. To see $\subseteq$,  let $w \in (x_1(x_1:x_2)):x_2$.  Then $x_2w \in x_1(x_1:x_2)$
and $x_2^2w \in x_1^2R$.  Say $x_2^2w = x_1^2v$.  Since $x_1$ is a Cohen-Macaulay multiplier, $v \in x_2^2:x_1^2 =x_2^2:x_1$, so $x_1v = x_2^2w'$, where $w' \in (x_1:x_2^2) = (x_1: x_2)$ (the last equality is by our choice for $x_2$).  Thus $x_2^2w = x_1 x^2_2 w'$, so $w - x_1 w' \in 0:x_2^2 = 0:x_2$.  Hence $x_2w = x_2x_1w' \in x_2x_1(x_1:x_2)$.

We also have that $(0:x_1) = (0:x_1x_2)$, since if $x_1x_2 z = 0$, then $x_1z \in (0: x_2) = H^0_\mm(R)$, and since $x_1$ is a Cohen-Macaulay multiplier, $x_1(x_1z) = 0$.  By our choice, $(0:x_1^2) = (0:x_1)$, so $z \in (0:x_1)$.

Consider the two short exact sequences
\begin{equation} 0 \longrightarrow \dfrac R {(x_1:x_2) + (0:x_1x_2)}  \overset{.x_1x_2}\longrightarrow\dfrac R {x_1x_2(x_1:x_2)} \longrightarrow \dfrac R {x_1x_2R} \longrightarrow 0
\end{equation}
and
\begin{equation} 0 \longrightarrow \dfrac R {(x_1:x_2) + (0:x_1)}  \overset{.x_1}\longrightarrow \dfrac R {x_1(x_1:x_2)} \longrightarrow \dfrac R {x_1R} \longrightarrow 0.
\end{equation}

The leftmost module in each of these sequences is the same, by the calculation above.  Also, since $(0:x_1x_2) = (0:x_1)$, the minimal resolutions of the rightmost modules are the same after the first step, so, by our choice of $x_1$ and $x_2$ in a large power of $\mm$, i.e., $x_1,x_2\in\mm^n\subset\ann(M)$ where $n$ is as in Proposition~\ref{extended}(2), we have $\sigma_i(M, R/x_1x_2R) = \sigma_i(M, R/x_1R)$.

Since  $x_1,x_2\in\mm^n\subset\ann(M)$ where $n$ is as in Proposition~\ref{extended}(2),  by Proposition \ref{extended}(1) and Remark \ref{addsigma}, we get $\sigma_i(M, R/(x_1x_2(x_1:x_2))) = \sigma_i(M, R/(x_1(x_1:x_2)))$.  Now,  applying additivity of $\sigma_i(M,-)$ to Equation~\ref{keySES}, we
obtain that $$\lambda(\syz_{i+1}M) = \sigma_i(M, R/(x_2, x_1(x_1:x_2))) = \sigma_i(M, R/x_2R). $$

$(2)$ By \cite[Lemma 4.2]{A}, $\lm(\syz_{i+1}M)=\infty$ for all $0\leq i\leq 1.$ Let $i=2.$ By Remark \ref{completion}, we may assume $R$ is a complete Noetherian local ring. Suppose, to the contrary, that $\lm(\syz_3M)<\infty.$  Then using Lemma \ref{add}, part $(1)$ of the theorem, and noting that $\lambda(M) \ge \lambda(M/H^0_\mm(R)M)$,  we get $$\lambda(\syz_{3}M) = \sigma_2(M, R/x_2R)
 = -\lambda(\tor_2(M, R/x_2R)) + \lambda (M/H^0_\mm(R)M) - \lambda (M) \le 0$$ which  contradicts that $\syz_3M\neq 0.$
\end{proof}

\begin{remark} Suppose that $(R,\mm)$ has dimension $d \ge 3$ and $M$ has finite length.  If $\syz_{i+1}(M)$ has finite length, it is entirely possible that, analogously to the dimension 2 case, there is a general element $x \in \mm$ such that $\lambda(\syz_{i+1}(M)) = \sigma_i(M, R/xR)$.  We have not attempted (yet) to prove this because even if we knew it to be true, when $i > 2$, we do not know how to conclude that $\sigma_i(M, R/xR) \le 0$.
\end{remark}

\begin{lemma}\label{divide}
Let $(R,\mathfrak m)$ be a  local ring of dimension $1$ or $2$, $\depth R= 0$ and $M$ be an $R$-module of finite length. Let $i\geq 2$ and $\lambda(\syz_{i+1}M)<\infty.$ Then $$\lm(\syz_{i+1}M)=\sum\limits_{j=0}^{i-2}(-1)^{i-j-1}\lm(\tor_{j}^R(M,H_{\mm}^0(R))).$$
	\end{lemma}
\begin{proof}
	By Remark \ref{completion}, we may assume $R$ is a complete Noetherian local ring. 
	
	 If $d=1$, let  $y$ be a  parameter of $R$ such that $(0:(y))=H^0_\mm(R).$   We can always do this by choosing $y$ not in any minimal prime of $R$ and such that $\mm^N \subseteq yR$, where $\mm^N H^0_\mm(R) = 0$.
	 
	 If $d=2$,  using Theorem \ref{dimension two}, choose a system of parameter $x_1,x_2$ such that  $x_2M=0,$ $(0:(x_2))=H^0_\mm(R)$ and $\lm(\syz_{i+1}M)=\sigma_i(M, R/x_2R).$ 
	 
	 Let $J=(y)$ if $d=1$,  and $J=(x_2)$ if $d=2.$ 
	 
	 Using Lemma \ref{add}, we get 
	\beqn\lm(\syz_{i+1}M)&=&\sigma_i(M,R/J)=\sigma_{i-1}(M,R/H_{\mm}^0(R)))+(-1)^{i+1}\lambda(M)
	\\&=&
	\sum\limits_{j=1}^{i-2}(-1)^{i-j-1}\lm(\tor_{j}^R(M,H_{\mm}^0(R)))+(-1)^{i-1}\big[\lm(M\otimes H_{\mm}^0(R))-\lm(M)\\&+&\lm(M/H_{\mm}^0(R)M)\big]+(-1)^{i}\lm(M/H_{\mm}^0(R)M)+(-1)^{i+1}\lm(M)
	\\&=&\sigma_{i-2}(M,H_{\mm}^0(R)).\eeqn 
	\end{proof}	
	
\begin{theorem}\label{bad to good}
	Let $(R,\mathfrak m)$ be a  local ring of dimension $1\leq d\leq 2$ and $\depth R= 0$.  Let $M$ be an $R$-module of finite length. Suppose one of the  following is true.
	\ben
	\item   Let $i\geq 3$ and $\lambda(\syz_{i}R/I)<\infty$ for some $\mm$-primary ideal $I\subset \mm^n$ where $n$ is as in Proposition \ref{extended}.
	\item   Let $i\geq 4$ and $\lambda(\syz_{i-2}H_{\mm}^0(R))<\infty.$
	\een
	Then $\lambda(\syz_{i+1}M)=\infty.$
\end{theorem}
\begin{proof} 
 Suppose $\lambda(\syz_{i+1}M)<\infty.$ Let $(G_\bullet,\delta_\bullet)$ be a minimal free resolution of $N$ where $N$ is either $R/I$ or  $H_{\mm}^0(R).$ Let $G_j=R^{{\beta_j}(N)}$ for all $0\leq j\leq i.$ If $\lambda(\syz_{k}(N))<\infty$ for $k\geq 2$ then $\sum\limits_{j=0}^{k-1}\beta_j(N)=0.$ For $k\geq 2,$  consider the complex of $R$-modules of finite length
$$ 0\longrightarrow\im(\delta_{k}\otimes 1_M)\longrightarrow G_{k-1}\otimes M\longrightarrow G_{k-2}\otimes M\longrightarrow\cdots\longrightarrow G_{1}\otimes M\longrightarrow G_{0}\otimes M\longrightarrow 0.$$ 

By Proposition \ref{length}, we get $(-1)^{k}\lm(\im(\delta_{k}\otimes 1_M))=\sum\limits_{j=0}^{k-1}(-1)^j\lm(\tor_{j}^R(M,N).$

 Consider $N=R/I$ and $k=i\geq 3.$ By  Proposition  \ref{extended}, $$\lm(\syz_{i+1}M)=-\lm(\tor_{i}^R(M,R/I) -\lm(\im(\delta_{i}\otimes 1_M))\leq 0$$ which is a contradiction.

 Consider $N=H_{\mm}^0(R)$ and $k=i-2\geq 2.$ By  Lemma \ref{divide}, $$\lm(\syz_{i+1}M)=-\lm(\tor_{i-2}^R(M,H_{\mm}^0(R)) -\lm(\im(\delta_{i-2}\otimes 1_M))\leq 0$$ which is a contradiction.
\end{proof}

\begin{theorem}
	Let $(R,\mathfrak m)$ be a Noetherian local ring of dimension $1\leq d\leq 2$ and $\depth R= 0.$ Suppose $I=(x_1,\ldots,x_d)$ is a parameter ideal. Then the following are true. 
	\ben
	\item If $d=1$  then $\lambda(\syz_{5}R/I)=\infty.$
	\item If $d=2,$ $R/H_{\mm}^0(R)$ is Cohen-Macaulay and for $i=1,2,$ $(0:x_i)=H_{\mm}^0(R)$ then $\lambda(\syz_{5}R/I)=\infty.$
	\een
\end{theorem}
\begin{proof}
$(1)$	Let $d=1.$ Since $I\cong R/(0:I),$ by \cite[Corollay 5.10]{SHB}, $\lambda(\syz_{5}R/I)=\lambda(\syz_{3}(0:I))=\infty.$
	
$(2)$ Let $d=2.$ Suppose $\lambda(\syz_{5}R/I)<\infty.$ From the long exact sequence of homology modules induced by the following short exact sequence of $R$-modules tensoring with $\otimes_R H_{\mm}^0(R),$
	$$ 0\longrightarrow R/(x_1)\cap(x_2)\longrightarrow R/(x_1)\oplus R/(x_2)\longrightarrow R/I\longrightarrow 0 ,$$ we get \beqn &&-\lm(\tor_{2}^R(R/I,H_{\mm}^0(R))+\lm(\tor_{1}^R(R/I,H_{\mm}^0(R))\\&\leq& -\lm(\tor_{1}^R(R/(x_1)\cap(x_2),H_{\mm}^0(R)))+\sum\limits_{k=1}^2{\lm(\tor_{1}^R(R/(x_k),H_{\mm}^0(R)))}.\eeqn
	Since $R/H_{\mm}^0(R)$ is Cohen-Macaulay, $\overline{x_1},\overline{x_2}$ forms a regular sequence in $R/H_{\mm}^0(R)$ where $``-"$ denotes the image in $R/H_{\mm}^0(R).$ Let $a\in (x_1)\cap(x_2)=x_2((x_1):(x_2)).$ Therefore $a=x_2r$ for some $r\in ((x_1):(x_2))$ and $\overline r\in ((\overline{x_1}):_{R/H_{\mm}^0(R)}(\overline{x_2}))=(\overline{x_1}).$ Since $x_2H_{\mm}^0(R)=0,$ we have $a=x_2r\in (x_1x_2).$  Now $(0:(x_1x_2))=H_{\mm}^0(R)$ implies $(x_1)\cap(x_2)=(x_1x_2)\cong R/H_{\mm}^0(R)\cong(x_1).$ Therefore by Lemmas \ref{add} and \ref{divide}, we get \beqn &&\lambda(\syz_{5}R/I)\\&=&-\lm(\tor_{2}^R(R/I,H_{\mm}^0(R))+\lm(\tor_{1}^R(R/I,H_{\mm}^0(R))-\lambda(R/I\otimes H_{\mm}^0(R))\\&\leq & -\lm(\tor_{1}^R(R/(x_1x_2),H_{\mm}^0(R)))+\sum\limits_{k=1}^2{\lm(\tor_{1}^R(R/(x_k),H_{\mm}^0(R)))}-\lambda(H_{\mm}^0(R))\\&=&-\lm(\tor_{2}^R(R/(x_1x_2),R/H_{\mm}^0(R)))+\lm(\tor_{2}^R(R/(x_1),R/H_{\mm}^0(R)))\\&+&\lm(\tor_{1}^R(R/(x_2),H_{\mm}^0(R)))-\lambda(H_{\mm}^0(R))=\lm(\tor_{1}^R(R/(x_2),H_{\mm}^0(R)))-\lambda(H_{\mm}^0(R)).\eeqn
	From the long exact sequence of homology modules induced by the following short exact sequence of $R$-modules tensoring with $\otimes_R H_{\mm}^0(R),$ $$0\longrightarrow R/(0:(x_2))\overset{.x_2}\longrightarrow R \longrightarrow R/(x_2)\longrightarrow 0$$ we get $\lm(\tor_{1}^R(R/(x_2),H_{\mm}^0(R)))=\lm(H_{\mm}^0(R)/(0:(x_2))H_{\mm}^0(R)).$ Therefore we get $\lambda(\syz_{5}R/I)\leq 0$ which is a contradiction. 
	\end{proof}
	
	We know of one other situation where we can assert that the higher syzygies (third or higher) of a finite length module are not finite length:
	
\begin{theorem}  Let $(R,\mm)$ be a local ring of dimension one and $M$ an $R$-module of finite length.  Suppose that for some $x \in \mm$ we have
$(0:_R x )= (0:_R \mm)$.  Then, for any even integer $i \ge 2$, if $\syz_{i+1}M\neq 0$ then $\lambda(\syz_{i+1}M)=\infty.$
\end{theorem}

\begin{proof}
If the statement fails, choose an example with the smallest possible (even) value of $i$. By \cite[Corollary 5.10]{SHB}, $\lambda(\syz_3 M) = \infty$, so $i > 2$.

 \wlg we may assume $i\geq 4.$ Let $(G_\bullet,\delta_\bullet)$ be a minimal free resolution of $M.$	Note that $$(*)\hspace{2mm}0\longrightarrow \syz_{i+1}\overset{\iota}\longrightarrow G_{i}\overset{\delta_{i}}\longrightarrow\cdots\longrightarrow G_1\overset{\delta_1}\longrightarrow G_0\overset{\delta_0}\longrightarrow 0$$ is an acyclic complex of finitely generated $R$-modules. Let $G_j=R^{\beta_j (M)}$ for all $0\leq j\leq i.$ Since $\lambda(M)<\infty$ and $\lambda(\syz_{i+1}(M))<\infty,$ we have $\sum\limits_{j=0}^i(-1)^{j}\beta_j(M)=0.$ 

By hypotheses, for $J = (0:x)R \cong k^t$ (for some integer $t$), we have the exact sequence $0 \to J \to R \overset{x}\to R \to R/xR \to 0$.  Thus for $i > 2$, $\tor_i(M, R/xR) \cong (\tor_{i-2}(M, k))^t \cong k^{t\beta_{i-2}(M)}$.  
We also have  the exact sequences
$$ 0 \to \tor_2(M, R/xR) \to J \otimes M \cong (M/\mm M)^t \to M \to M/JM \to 0
$$
and
$$
0 \to \tor_1(M, R/xR) \to M/JM \to M \to M/xM \to 0
$$
Therefore, using the length computations from the above sequences and noting a lot of cancellation,
\begin{align*}\sigma_i(M, R/xR) &= \sum_{j=3}^i (-1)^{i-j+1}\lambda(k^{t\beta_{j-2}(M)}) +(-1)^{i-2+1} \lambda(\tor_2(M, R/xR)) \\
& \qquad + (-1)^{i-1+1} \lambda(\tor_1(M, R/xR) + (-1)^{i-0+1} \lambda (M/xM) \\
                 & = -t\beta_{i-2}(M) + t\beta_{i-3}(M)+ \cdots (-1)^i t\beta_1(M) +(-1)^{i-2+1} t\beta_0(M)\\
                 &= t(\beta_i (M)- \beta_{i-1}(M)).
 \end{align*}
 Since $\syz_{i+1} M$ has finite length, we have $\beta_i(M) \le \beta_{i-1}(M)$ and hence by Proposition \ref{length}, $$0\leq \lm(\im(\delta_{i+1}\otimes 1_{R/xR}))=\sigma_i(M,R/xR)=t(\beta_i(M) - \beta_{i-1}(M)) \le 0.$$ 
 Therefore $\beta_i(M)= \beta_{i-1}(M)$, which implies $\lambda(\syz_{i-1}M)<\infty$, contradicting our choice of $i$ smallest.
 \end{proof}


\begin{thebibliography}{99}
	\bibliographystyle{alpha}

\bibitem{Ab} I. M. Aberbach, {\em Finite phantom projective dimension}, Amer. J. Math. {\bf 116} (1994), no. 2, 447-477. 

	\bibitem{AL} I. M. Aberbach and J. Li, {\em Asymptotic vanishing conditions which force regularity in local rings of prime characteristic}, Math. Res. Lett. {\bf 15} (2008), 815-820.
	
	\bibitem{A} M. Asgharzadeh, {\em on the dimension of  syzygies}, \arxiv{1705.04952v3}.
	
	\bibitem{BE} D. A. Buchsbaum and D. Eisenbud, {\em What makes a complex exact ?}, J. Algebra {\bf 25}, (1973), 259-268.
	
   \bibitem{CHKV} A. Corso, C. Huneke, D. Katz and W. V. Vasconcelos, {\em Integral closure of ideals and annihilators of homology}, in Commutative algebra, Lect. Notes Pure Appl. Math. {\bf 244}, 33-48.
	
	\bibitem{BH} W. Bruns and J. Herzog,{\em Cohen-Macaulay Rings}, Cambridge University Press, Cambridge, 1993. 
	
	\bibitem{SHB} A. De Stefani, C. Huneke and L. N\'{u}\~{n}ez-Betancourt, {\em Frobenius betti numbers and modules of finite projective dimension}, J. Commut. Alg., (2017), no. 4, 455-490.
	
	\bibitem{D} S. P. Dutta, {\em Ext and Frobenius}, J. Algebra, {\bf 127}, (1989), 163-177. 
	
	\bibitem{HH} M. Hochster and C. Huneke, {\em Tight closure, invariant theory and the Briancon-Skoda theorem}, J. Amer. Math. Soc. {\bf 3} (1990), no. 1, 31-116.
	
	\bibitem{HH92} M. Hochster and C. Huneke, {\em Infinite integral extensions and big Cohen-Macaulay algebras}, Annals of Math.
	{\bf 135} (1992), 53-89.
	
	\bibitem{HH2} M. Hochster and C. Huneke, {\em F-regularity, test elements, and smooth base change}, Trans. Amer. Math. Soc. {\bf 346} (1994), no. 1, 1-62.
		
	\bibitem{Mi} C. Miller, {\em A Frobenius characterization of finite projective dimension over complete intersections}, Math. Z. {\bf 233} (2000), 127-136.
	
	\bibitem{PS} C. Peskine and L. Szpiro, {\em Dimension projective finie et cohomologie locale, Applications \`{a} la d\'{e}monstration de conjectures de M. Auslander, H. Bass et A. Grothendieck}, Inst. Hautes Etud, Sci. Publ. Math. {\bf 42} (1973), 47-119.
	
	\bibitem{R} P. Roberts, {\em Multiplicities and Chern Classes in Local Algebra}, Cambridge University Press, 1998.
	
	\bibitem{Ro87} P. Roberts, {\em Le th\'{e}or\`eme d'intersection}, C. R. Acad. Sci. Paris S\'{e}r. I Math. {\bf 304} (1987), 177--180.
	
	\bibitem{S} G. Seibert, {\em Complexes with homology of finite length and Frobenius functors}, J. Algebra {\bf 125} (1989), 278-287.
	
\end{thebibliography}
\end{document}